\documentclass[11pt]{article}
       \usepackage{amsfonts}
       \usepackage{stmaryrd}
       \usepackage{latexsym,amssymb,mathrsfs,fancyhdr}
       \font\tenmsb=msbm10
       \font\sevenmsb=msbm7
       \font\fivemsb=msbm5
       \catcode`\@=11
       \ifx\amstexloaded@\relax\catcode`\@=\active
       \endinput\else\let\amstexloaded@\relax\fi
       \def\spaces@{\space\space\space\space\space}
       \def\spaces@@{\spaces@\spaces@\spaces@\spaces@\spaces@}
       \def\space@.  {\futurelet\space@\relax}
       \space@.   %
       \def\Err@#1{\errhelp\defaulthelp@\errmessage{AmS-teX error: #1}}
       \def\relaxnext@{\let\next\relax}
       \def\accentfam@{7}
       \def\noaccents@{\def\accentfam@{0}}
       \def\Cal{\relaxnext@\ifmmode\let\next\Cal@\else
       \def\next{\Err@{Use \string\Cal\space only in math mode}}\fi\next}
       \def\Cal@#1{{\Cal@@{#1}}}
       \def\Cal@@#1{\noaccents@\fam\tw@#1}
       \def\Bbb{\relaxnext@\ifmmode\let\next\Bbb@\else
       \def\next{\Err@{Use \string\Bbb\space only in math mode}}\fi\next}
       \def\Bbb@#1{{\Bbb@@{#1}}}
       \def\Bbb@@#1{\noaccents@\fam\msbfam#1}
       \def\co{\tiny{\textcircled{\tiny\#}}}
       \newfam\msbfam
       \textfont\msbfam=\tenmsb
       \scriptfont\msbfam=\sevenmsb
       \scriptscriptfont\msbfam=\fivemsb
\usepackage{dsfont}
\usepackage[german,english]{babel}
\usepackage{amsmath,amssymb}
\usepackage[square, comma, sort&compress, numbers]{natbib}
\usepackage{cases}
\usepackage{mathrsfs}
\usepackage{amsmath}
\usepackage{amsthm}
\usepackage{amsfonts}
\usepackage{amssymb}
\usepackage{latexsym}
\usepackage{fancyhdr}
\usepackage{extarrows}

\newtheorem{thm}{Theorem}[section]

\newtheorem{lem}[thm]{Lemma}
\newtheorem{rem}[thm]{Remark}
\newtheorem{iteration lemma}[thm]{iteration Lemma}
\newtheorem{cor}[thm]{Corollary}

\newtheorem*{acknowledgements*}{ACKNOWLEDGEMENtS}

\begin{document}

\setlength{\columnsep}{5pt}
\title{\bf The core and dual core inverse of a morphism with factorization}
\author{Tingting  Li\footnote{ E-mail: littnanjing@163.com},
\ Jianlong Chen\footnote{ Corresponding author. E-mail: jlchen@seu.edu.cn}\\
School of  Mathematics, Southeast University \\  Nanjing 210096,  China }
     \date{}

\maketitle
\begin{quote}
{\textbf{}\small
Let $\mathscr{C}$ be a category with an involution $\ast$.
Suppose that
$\varphi : X \rightarrow X$ is a morphism and $(\varphi_1, Z, \varphi_2)$ is an (epic, monic) factorization of $\varphi$ through $Z$,
then $\varphi$ is core invertible if and only if $(\varphi^{\ast})^2\varphi_1$ and $\varphi_2\varphi_1$ are both left invertible if and only if
$((\varphi^{\ast})^2\varphi_1, Z, \varphi_2)$,
$(\varphi_2^{\ast}, Z, \varphi_1^{\ast}\varphi^{\ast}\varphi)$ and $(\varphi^{\ast}\varphi_2^{\ast}, Z, \varphi_1^{\ast}\varphi)$ are all essentially unique (epic, monic) factorizations of $(\varphi^{\ast})^2\varphi$ through $Z$.
We also give the corresponding result about dual core inverse.
In addition,
we give some characterizations about the coexistence of core inverse and dual core inverse of an $R$-morphism in the category of $R$-modules of a given ring $R$.


}
\end{quote}

\section{ Introduction }\label{a}
Let $\mathscr{C}$ be a category.
$\mathscr{C}$ is said to have an involution $\ast$ provided that there is a unary operation $\ast$ on the morphisms such that
$\varphi : X \rightarrow Y$ implies $\varphi^{\ast} : Y \rightarrow X$ and that $(\varphi^{\ast})^{\ast}=\varphi, (\varphi\psi)^{\ast}=\psi^{\ast}\varphi^{\ast}$ for all morphisms $\varphi$ and $\psi$ in $\mathscr{C}$ .
(See, for example, \cite[p. 131]{PR2}.)
Let $\varphi : X \rightarrow Y$ and $\chi : Y \rightarrow X$ be morphisms of $\mathscr{C}$.
Consider the following four equations:
\begin{center}
  $(1)$ $\varphi\chi\varphi=\varphi$,~~~$(2)$ $\chi\varphi\chi=\chi$,~~~$(3)$ $(\varphi\chi)^{\ast}=\varphi\chi$,~~~$(4)$ $(\chi\varphi)^{\ast}=\chi\varphi$.
\end{center}
Let $\varphi\{i,j,\cdots,l\}$ denote the set of morphisms $\chi$ which satisfy equations $(i),(j),\cdots,(l)$ from among equations $(1)$-$(4)$,
and in this case,
$\chi$ is called the $\{i,j,\cdots,l\}$-inverse of $\varphi$.
If $\chi \in \varphi\{1, 3\}$,
then $\chi$ is called a $\{1, 3\}$-inverse of $\varphi$ and is denoted by $\varphi^{(1, 3)}$.
A $\{1, 4\}$-inverse of $\varphi$ can be similarly defined.
If $\chi \in \varphi\{1, 2, 3, 4\}$,
then $\chi$ is called the Moore-Penrose inverse of $\varphi$.
If such a $\chi$ exists,
then it is unique and denoted by $\varphi^{\dagger}$.
If $X=Y$,
then a morphism $\varphi$ is group invertible if there is a morphism $\chi \in \varphi\{1, 2\}$ that commutes with $\varphi$.
If the group inverse of $\varphi$ exists,
then it is unique and denoted by $\varphi^{\#}$.
References to group inverses and Moore-Penrose inverses of morphisms can be seen in,
for example,
\cite{PR1}-\cite{Liu3}.

In 2010,
O.M. Baksalary and G. Trenkler introduced the core and dual core inverse of a complex matrix in \cite{OM}.
Raki\'{c} et al. \cite{DSR} generalized core inverse of a complex matrix to the case of an element in a ring with an involution $\ast$,
and they use five equations to characterize the core inverse.
In the following,
we rewrite these five equations in the category case.
Let $\mathscr{C}$ be a category with an involution and $\varphi : X \rightarrow X$ a morphism of $\mathscr{C}$.
If there is a morphism $\chi : X \rightarrow X$ satisfying
\begin{equation}\label{vvv}
\begin{split}
  \varphi\chi\varphi=\varphi,~\chi\varphi\chi=\chi,~(\varphi\chi)^{\ast}=\varphi\chi,~\varphi\chi^2=\chi,~\chi\varphi^2=\varphi,
\end{split}
\end{equation}
then $\varphi$ is core invertible and $\chi$ is called the core inverse of $\varphi$.
If such $\chi$ exists,
then it is unique and denoted by $\varphi^{\co}$.
In \cite{XSZ},
Xu et al. proved that equations $\varphi\chi\varphi=\varphi$ and $\chi\varphi\chi=\chi$ in (\ref{vvv}) can be dropped,
that is to say,
$\varphi$ is core invertible with $\varphi^{\co}=\chi$ if and only if
\begin{equation*}
\begin{split}
  (\varphi\chi)^{\ast}=\varphi\chi,~\varphi\chi^2=\chi,~\chi\varphi^2=\varphi.
\end{split}
\end{equation*}
And the dual core inverse can be given dually and denoted by $\varphi_{\co}$.
References to core and dual core inverses of morphisms can be seen in,
for example,
\cite{LiChen}.

In this paper,
the convention is used of reading morphism composition from left to right,
that is to say,
$$\varphi\psi : X \xrightarrow{\mbox{$\begin{array}{c}\varphi\end{array}$}} Y \xrightarrow{\mbox{$\begin{array}{c}\psi\end{array}$}} Z.$$
A morphism $\varphi$ is said to be epic if $\varphi\psi=\varphi\psi'$ implies $\psi=\psi'$,
and monic if $\psi\varphi=\psi'\varphi$ implies $\psi=\psi'$.
A morphism $\varphi : X \rightarrow Y$ is left invertible if there exists a morphism $\psi : Y \rightarrow X$ such that $\psi\varphi=1_Y$,
and right invertible if there exists a morphism $\psi : Y \rightarrow X$ such that $\varphi\psi=1_X$.

Let $\varphi : X \rightarrow Y$ be a morphism in $\mathscr{C}$.
If $\varphi_1 : X \rightarrow Z$ and $\varphi_2 : Z \rightarrow Y$ are morphisms and $\varphi=\varphi_1\varphi_2$,
then $(\varphi_1, Z, \varphi_2)$ is called a factorization of $\varphi$ through an object $Z$.
A factorization $(\varphi_1, Z, \varphi_2)$ of $\varphi$ through $Z$ is called an (epic, monic) factorization of $\varphi$ whenever $\varphi_1$ is epic and $\varphi_2$ is monic.
Furthermore,
an (epic, monic) factorization $(\varphi_1, Z, \varphi_2)$ of $\varphi$ through $Z$ is said to be essentially unique (see, for example, \cite{PR2}) if whenever $(\varphi_1', Z', \varphi_2')$ is also an (epic, monic) factorization of $\varphi$ through an object $Z'$,
then there is an invertible morphism $\nu : Z \rightarrow Z'$ such that $\varphi_1\nu=\varphi_1'$ and $\nu\varphi_2'=\varphi_2$.
References to generalized inverses of a factorization can be seen in,
for example,
\cite{PPP} and \cite{HHZhu}.

In \cite{PR2},
D.W. Robinson and R. Puystjens showed us the characterizations about the Moore-Penrose inverse of a morphism with a factorization.
And in \cite{DW2},
R. Puystjens and D.W. Robinson gave the characterizations about the group inverse of a morphism with a factorization,
and they also gave the characterizations about the Moore-Penrose inverse which is different from the results in \cite{PR2}.
Inspired by them,
the second part in this paper will give some characterizations about the core inverse and the dual core inverse of a morphism with a factorization,
respectively.

In \cite{Chen} and \cite{Li},
authors investigated the coexistence of core inverse and dual core inverse of an element in a $\ast$- ring which is a ring with an involution $\ast$ provided that there is an anti-isomorphism $\ast$ such that
$(a^{\ast})^{\ast}=a, (a+b)^{\ast}=a^{\ast}+b^{\ast}$ and $(ab)^{\ast}=b^{\ast}a^{\ast}$ for all $a, b\in R$.
It makes sense to investigate the coexistence of core inverse and dual core inverse of an $R$-morphism in the category of $R$-modules of a given ring $R$.
We give some characterizations about the coexistence of core inverse and dual core inverse of an $R$-morphism in the third part.

The following notations will be used in this paper:
$aR=\{ax~|~x\in R\}$, $Ra=\{xa~|~x\in R\}$, $^{\circ}\!a=\{x\in R~|~xa=0\}$, $a^{\circ}=\{x\in R~|~ax=0\}$.
In addition,
some auxiliary lemmas and results are presented for the further reference.

\begin{lem} \cite[Lemma $1$]{PR2}\label{ess.uniq}
Let $\varphi : X \rightarrow Y$ be a morphism of $\mathscr{C}$ with a factorization $(\varphi_1, Z, \varphi_2)$ through an object $Z$.
If $\varphi_1 : X \rightarrow Z$ is left invertible and $\varphi_2 : Z \rightarrow Y$ is right invertible,
then $(\varphi_1, Z, \varphi_2)$  is an essentially unique (epic, monic) factorization of $\varphi$ through $Z$.
\end{lem}

It should be pointed out,
\cite[Theorem $2.6$ and $2.8$]{XSZ},
\cite[Theorem $2.10$]{Li},
\cite[Theorem $2.11$]{Li} and \cite[p. $201$]{Hartwig} were put forward in a $\ast$- ring.
Actually,
one can easily prove that these results are also valid in a category with an involution $\ast$.
Thus,
we can rewrite them in the category case as the following Lemma~\ref{222} - \ref{000},
respectively.

\begin{lem} \cite[Theorem $2.6$ and $2.8$]{XSZ}\label{222}
Let $\varphi : X \rightarrow X$ be a morphism of $\mathscr{C}$,
we have the following results:\\
(i) $\varphi$ is core invertible if and only if $\varphi$ is group invertible and $\{1,3\}$-invertible.
In this case,
$\varphi^{\co}=\varphi^{\#}\varphi\varphi^{(1,3)}$.\\
(ii) $\varphi$ is dual core invertible if and only if $\varphi$ is group invertible and $\{1,4\}$-invertible.
In this case,
$\varphi_{\co}=\varphi^{(1,4)}\varphi\varphi^{\#}$.
\end{lem}

\begin{lem}\cite[Theorem $2.10$]{Li}\label{Li1}
Let $\varphi : X \rightarrow X$ be a morphism of $\mathscr{C}$ and $n\geqslant 2$ a positive integer,
we have the following results:\\
(i) $\varphi$ is core invertible if and only if there exist morphisms $\varepsilon : X \rightarrow X$ and $\tau : X \rightarrow X$ such that $\varphi=\varepsilon(\varphi^{\ast})^n\varphi=\tau\varphi^n$.
In this case,
$\varphi^{\co}=\varphi^{n-1}\varepsilon^{\ast}$.\\
(ii) $\varphi$ is dual core invertible if and only if there exist morphisms $\theta : X \rightarrow X$ and $\rho : X \rightarrow X$ such that $\varphi=\varphi(\varphi^{\ast})^n\theta=\varphi^n\rho$.
In this case,
$\varphi_{\co}=\theta^{\ast}\varphi^{n-1}$.
\end{lem}

\begin{lem}\cite[Theorem $2.11$]{Li}\label{Li2}
Let $\varphi : X \rightarrow X$ be a morphism of $\mathscr{C}$ and $n\geqslant 2$ a positive integer,
the following statements are equivalent:\\
(i) $\varphi$ is both Moore-Penrose invertible and group invertible.\\
(ii) $\varphi$ is both core invertible and dual core invertible.\\
(iii) There exist $\alpha : X \rightarrow X$ and $\beta : X \rightarrow X$ such that $\varphi=\alpha(\varphi^{\ast})^n\varphi=\varphi(\varphi^{\ast})^n\beta$.\\
In this case,
\begin{equation*}
\begin{split}
    \varphi^{\co}&~=\varphi^{n-1}\alpha^{\ast},\\
    \varphi_{\co}&~=\beta^{\ast}\varphi^{n-1},\\
    \varphi^{\dagger}&~=\beta\varphi^{2n-1}\alpha^{\ast},\\
    \varphi^{\#}&~=(\varphi^{n-1}\alpha^{\ast})^2\varphi=\varphi(\beta^{\ast}\varphi^{n-1})^2.
\end{split}
\end{equation*}
\end{lem}

\begin{lem} \cite[p. $201$]{Hartwig}\label{000}
Let $\varphi : X \rightarrow Y$ be a morphism of $\mathscr{C}$,
we have the following results:\\
$(i)$ $\varphi$ is $\{1,3\}$-invertible with $\{1,3\}$-inverse $\chi : Y \rightarrow X$ if and only if $\chi^{\ast}\varphi^{\ast}\varphi=\varphi;$\\
$(ii)$ $\varphi$ is $\{1,4\}$-invertible with $\{1,4\}$-inverse $\zeta : Y \rightarrow X$ if and only if $\varphi\varphi^{\ast}\zeta^{\ast}=\varphi.$
\end{lem}

\section{Main Results}\label{a}
In \cite{PR2},
D.W. Robinson and R. Puystjens gave some results about the Moore-Penrose inverse of a morphism with a factorization.
And in \cite{DW2},
R. Puystjens and D.W. Robinson gave the characterizations about the group inverse of a morphism with a factorization,
as follows.

\begin{lem}\cite[Theorem~$1$]{DW2}\label{KKK}
Let $(\varphi_1, Z, \varphi_2)$ be an (epic, monic) factorization of a morphism $\varphi : X \rightarrow X$ through an object $Z$ of a category.
Then the following statements are equivalent:\\
$(i)$ $\varphi$ has a group inverse $\varphi^{\#} : X \rightarrow X$,\\
$(ii)$ $\varphi_2\varphi_1 : Z \rightarrow Z$ is invertible,\\
$(iii)$ $(\varphi_1, Z, \varphi_2\varphi)$ and $(\varphi\varphi_1, Z, \varphi_2)$ are both essentially unique (epic, monic) factorizations of $\varphi^2$ through $Z$.
\end{lem}

\begin{lem}\cite[Theorem~$2$ and Theorem~$3$]{PR2}\label{kkkk}
Let $\varphi : X \rightarrow Y$ be a morphism of a category with involution $\ast$.
If $(\varphi_1, Z, \varphi_2)$ is an (epic, monic) factorization of $\varphi$ through $Z$,
then the following statements are equivalent: \\
$(i)$ $\varphi$ has a Moore-Penrose inverse with respect to $\ast$,\\
$(ii)$ $\varphi^{\ast}\varphi_1$ is left invertible and $\varphi_2^{\ast}\varphi$ is right invertible,\\
$(iii)$ $(\varphi^{\ast}\varphi_1, Z, \varphi_2)$ and $(\varphi_1, Z, \varphi_2\varphi^{\ast})$ are,
respectively,
essentially unique (epic, monic) factorizations of $\varphi^{\ast}\varphi$ and $\varphi\varphi^{\ast}$ through $Z$,\\
$(iv)$ $\varphi_1^{\ast}\varphi_1$ and $\varphi_2\varphi_2^{\ast}$ are both invertible.\\
In this case,
$$\varphi^{\dagger}=\varphi_2^{\ast}(\varphi_2\varphi_2^{\ast})^{-1}(\varphi_1^{\ast}\varphi_1)^{-1}\varphi_1^{\ast}.$$
\end{lem}

From Lemma~\ref{kkkk},
we know that $\varphi$ has a Moore-Penrose inverse if and only if
$\varphi^{\ast}\varphi_1$ is left invertible and $\varphi_2\varphi^{\ast}$ is right invertible if and only if
both $\varphi_1^{\ast}\varphi_1$ and $\varphi_2\varphi_2^{\ast}$ are invertible.
Thus we can easily prove the following two lemmas in a similar way.

\begin{lem}\label{13inv.}
Let $\varphi : X \rightarrow Y$ be a morphism of a category with involution $\ast$.
If $(\varphi_1, Z, \varphi_2)$ is an (epic, monic) factorization of $\varphi$ through $Z$,
then $\varphi$ has a $\{1,3\}$-inverse with respect to $\ast$ if and only if $\varphi^{\ast}\varphi_1 : Y \rightarrow Z$ is left invertible.
\end{lem}

\begin{proof}
Since
\begin{equation*}
\begin{split}
    \varphi_1\cdot1_Z\cdot\varphi_2
    &~=\varphi = \varphi\varphi^{(1,3)}\varphi\varphi^{(1,3)}\varphi = \varphi\varphi^{(1,3)}(\varphi\varphi^{(1,3)})^{\ast}\varphi \\
    &~=\varphi\varphi^{(1,3)}(\varphi^{(1,3)})^{\ast}\varphi^{\ast}\varphi = \varphi_1\varphi_2\varphi^{(1,3)}(\varphi^{(1,3)})^{\ast}\varphi^{\ast}\varphi_1\varphi_2,
\end{split}
\end{equation*}
$\varphi_1$ is epic and $\varphi_2$ monic,
then
\begin{equation*}
\begin{split}
\varphi_2\varphi^{(1,3)}(\varphi^{(1,3)})^{\ast}\varphi^{\ast}\varphi_1=1_Z.
\end{split}
\end{equation*}
Therefore,
$\varphi^{\ast}\varphi_1$ is left invertible.

Conversely,
there is a morphism $\mu : Z \rightarrow Y$ such that $\mu\varphi^{\ast}\varphi_1=1_Z$,
then
$$\varphi=\varphi_1\cdot1_Z\cdot\varphi_2=\varphi_1(\mu\varphi^{\ast}\varphi_1)\varphi_2=\varphi_1\mu\varphi^{\ast}\varphi,$$
thus $\varphi$ is $\{1,3\}$-invertible by Lemma~\ref{000}.
\end{proof}

Similarly,
we have a dual result for $\{1,4\}$-inverse.

\begin{lem}\label{14inv.}
Let $\varphi : X \rightarrow Y$ be a morphism of a category with involution $\ast$.
If $(\varphi_1, Z, \varphi_2)$ is an (epic, monic) factorization of $\varphi$ through $Z$,
then $\varphi$ has a $\{1,4\}$-inverse with respect to $\ast$ if and only if $\varphi_2\varphi^{\ast} : Z \rightarrow X$ is right invertible.
\end{lem}

Inspired by D.W. Robinson and R. Puystjens \cite{DW2},
we get some characterizations of the core invertibility of a morphism with an (epic, monic) factorization in a category $\mathscr{C}$.

\begin{thm} \label{co-inv.}
Let $\varphi : X \rightarrow X$ be a morphism of a category with involution $\ast$.
If $(\varphi_1, Z, \varphi_2)$ is an (epic, monic) factorization of $\varphi$ through $Z$,
then the following statements are equivalent: \\
(i) $\varphi$ has a core inverse with respect to $\ast$;\\
(ii) $\varphi^{\ast}\varphi_1 : X \rightarrow Z$ is left invertible and $\varphi_2\varphi_1 : Z \rightarrow Z$ is invertible;\\
(iii) $(\varphi^{\ast})^n\varphi_1 : X \rightarrow Z$ and $\varphi_2\varphi_1 : Z \rightarrow Z$ are both left invertible for any positive integer $n\geqslant 2$;\\
(iv) $((\varphi^{\ast})^2\varphi_1, Z, \varphi_2)$,
$(\varphi_2^{\ast}, Z, \varphi_1^{\ast}\varphi^{\ast}\varphi)$ and $(\varphi^{\ast}\varphi_2^{\ast}, Z, \varphi_1^{\ast}\varphi)$ are all essentially unique (epic, monic) factorizations of $(\varphi^{\ast})^2\varphi$ through $Z$.
\end{thm} \label{core-inverse-1}

\begin{proof}
$(i) \Leftrightarrow (ii).$
By Lemma~\ref{222},
$\varphi$ is core invertible if and only if $\varphi$ is group invertible and $\{1,3\}$-invertible.
Moreover,
$\varphi$ is group invertible if and only $\varphi_2\varphi_1 : Z \rightarrow Z$ is invertible by Lemma~\ref{KKK},
and $\varphi$ is $\{1,3\}$-invertible if and only if $\varphi^{\ast}\varphi_1 : X \rightarrow Z$ is left invertible by Lemma~\ref{13inv.}.
In conclusion,
$\varphi$ is core invertible if and only if $\varphi^{\ast}\varphi_1 : X \rightarrow Z$ is left invertible and $\varphi_2\varphi_1 : Z \rightarrow Z$ is invertible.

$(i) \Rightarrow (iii).$
Since
\begin{equation*}
\begin{split}
    \varphi_1\cdot1_Z\cdot\varphi_2
    &~=\varphi = \varphi\varphi^{\co}\varphi = \varphi(\varphi^{\co}\varphi\varphi^{\co})\varphi = \varphi\varphi^{\co}(\varphi\varphi^{\co})^{\ast}\varphi\\
    &~=\varphi\varphi^{\co}(\varphi^{\co})^{\ast}\varphi^{\ast}\varphi = \varphi\varphi^{\co}(\varphi\varphi^{\co}\varphi^{\co})^{\ast}\varphi^{\ast}\varphi\\
    &~=\varphi_1\varphi_2\varphi^{\co}(\varphi^{\co})^{\ast}(\varphi^{\co})^{\ast}(\varphi^{\ast})^2\varphi_1\varphi_2\\
    &~=\varphi_1\varphi_2\varphi^{\co}(\varphi^{\co})^{\ast}(\varphi\varphi^{\co}\varphi^{\co})^{\ast}(\varphi^{\ast})^2\varphi_1\varphi_2\\
    &~=\varphi_1\varphi_2\varphi^{\co}(\varphi^{\co})^{\ast}((\varphi^{\co})^2)^{\ast}(\varphi^{\ast})^3\varphi_1\varphi_2\\
    &~=\cdots\\
    &~=\varphi_1\varphi_2\varphi^{\co}(\varphi^{\co})^{\ast}((\varphi^{\co})^{n-1})^{\ast}(\varphi^{\ast})^n\varphi_1\varphi_2,
\end{split}
\end{equation*}
$\varphi_1$ is epic and $\varphi_2$ is monic,
then
\begin{equation*}
\begin{split}
\varphi_2\varphi^{\co}(\varphi^{\co})^{\ast}((\varphi^{\co})^{n-1})^{\ast}(\varphi^{\ast})^n\varphi_1=1_Z,
\end{split}
\end{equation*}
thus $(\varphi^{\ast})^n\varphi_1 : X \rightarrow Z$ is left invertible for any $n\geqslant 2$.

Similarly,
from
\begin{equation*}
\begin{split}
    \varphi_1\cdot1_Z\cdot\varphi_2
    &~=\varphi = \varphi\varphi^{\co}\varphi = \varphi\varphi^{\co}(\varphi^{\co}\varphi^2) = \varphi_1\varphi_2\varphi^{\co}\varphi^{\co}\varphi_1\varphi_2\varphi_1\varphi_2,
\end{split}
\end{equation*}
$\varphi_1$ is epic and $\varphi_2$ is monic,
we obtain
\begin{equation}\label{D1}
\begin{split}
\varphi_2\varphi^{\co}\varphi^{\co}\varphi_1\varphi_2\varphi_1=1_Z,
\end{split}
\end{equation}
thus $\varphi_2\varphi_1 : Z \rightarrow Z$ is left invertible.

$(iii) \Rightarrow (i).$
Suppose that $(\varphi^{\ast})^n\varphi_1 : X \rightarrow Z$ and $\varphi_2\varphi_1 : Z \rightarrow Z$ are both left invertible for any $n\geqslant 2$,
then there exist $\mu : Z \rightarrow X$ and $\nu : Z \rightarrow Z$ such that
\begin{equation*}
\begin{split}
\mu(\varphi^{\ast})^n\varphi_1 = 1_Z = \nu\varphi_2\varphi_1.
\end{split}
\end{equation*}
Therefore,
\begin{equation*}
\begin{split}
\varphi
&~= \varphi_1\cdot1_Z\cdot\varphi_2 = \varphi_1(\nu\varphi_2\varphi_1)\varphi_2 = \varphi_1\nu(\nu\varphi_2\varphi_1)\varphi_2\varphi_1\varphi_2 = \varphi_1\nu^2\varphi_2\varphi^2\\
&~= \varphi_1\nu^2(\nu\varphi_2\varphi_1)\varphi_2\varphi^2 = \varphi_1\nu^3\varphi_2\varphi^3 = \cdots =\varphi_1\nu^n\varphi_2\varphi^n,
\end{split}
\end{equation*}
\begin{equation*}
\begin{split}
\varphi = \varphi_1\cdot1_Z\cdot\varphi_2 = \varphi_1(\mu(\varphi^{\ast})^n\varphi_1)\varphi_2 = \varphi_1\mu(\varphi^{\ast})^n\varphi.
\end{split}
\end{equation*}
Thus $\varphi$ has a core inverse with $\varphi^{\co}=\varphi^{n-1}(\varphi_1\mu)^{\ast}$ by Lemma~\ref{Li1}.

$((i) \Leftrightarrow (iii)) \Rightarrow (iv).$
When taking $n=2$,
$\varphi$ has a core inverse if and only if $(\varphi^{\ast})^2\varphi_1 : X \rightarrow Z$ and $\varphi_2\varphi_1 : Z \rightarrow Z$ are both left invertible.
To begin with,
since $(\varphi^{\ast})^2\varphi_1$ is left invertible,
and equality (\ref{D1}) shows that $\varphi_2$ is right invertible,
thus $((\varphi^{\ast})^2\varphi_1, Z, \varphi_2)$ is an essentially unique (epic, monic) factorization of $(\varphi^{\ast})^2\varphi$ through $Z$ by Lemma~\ref{ess.uniq}.

Next,
$\varphi_2$ is right invertible,
which gives that $(\varphi_2)^{\ast}$ is left invertible.
In addition,
equality (\ref{D1}) gives
\begin{equation*}
\begin{split}
   1_Z = \varphi_2\varphi^{\co}\varphi^{\co}\varphi_1\varphi_2\varphi_1 = \varphi_2\varphi^{\co}(\varphi^{\co}\varphi\varphi^{\co})\varphi\varphi_1 = \varphi_2\varphi^{\co}\varphi^{\co}(\varphi^{\co})^{\ast}\varphi^{\ast}\varphi\varphi_1,
\end{split}
\end{equation*}
so $\varphi^{\ast}\varphi\varphi_1$ is left invertible,
that is to say,
$\varphi_1^{\ast}\varphi^{\ast}\varphi$ is right invertible.
Therefore,
$(\varphi_2^{\ast}, Z, \varphi_1^{\ast}\varphi^{\ast}\varphi)$ is an essentially unique (epic, monic) factorization of $(\varphi^{\ast})^2\varphi$ through $Z$ by Lemma~\ref{ess.uniq}.

Finally,
since $(\varphi^{\ast})^2\varphi_1$ is left invertible,
then $\varphi^{\ast}\varphi_1$ is also left invertible,
which implies that $\varphi_1^{\ast}\varphi$ is right invertible.
In addition,
equality (\ref{D1}) shows
\begin{equation*}
\begin{split}
   1_Z = \varphi_2\varphi^{\co}\varphi^{\co}\varphi_1\varphi_2\varphi_1 = \varphi_2(\varphi\varphi^{\co}\varphi^{\co})\varphi^{\co}\varphi\varphi_1,
\end{split}
\end{equation*}
thus $\varphi_2\varphi$ is right invertible,
that is to say,
$\varphi^{\ast}\varphi_2^{\ast}$ is left invertible.
Hence $(\varphi^{\ast}\varphi_2^{\ast}, Z, \varphi_1^{\ast}\varphi)$ is an essentially unique (epic, monic) factorization of $(\varphi^{\ast})^2\varphi$ through $Z$ by Lemma~\ref{ess.uniq}.

$(iv) \Rightarrow (iii).$
Suppose that $((\varphi^{\ast})^2\varphi_1, Z, \varphi_2)$,
$(\varphi^{\ast}\varphi_2^{\ast}, Z, \varphi_1^{\ast}\varphi)$ and $(\varphi_2^{\ast}, Z, \varphi_1^{\ast}\varphi^{\ast}\varphi)$ are essentially unique (epic, monic) factorizations of $(\varphi^{\ast})^2\varphi$ through $Z$.
In particular,
there exist invertible morphisms $\rho: Z \rightarrow Z, \sigma : Z \rightarrow Z$ such that
\begin{equation*}
\begin{split}
   (\varphi^{\ast})^2\varphi_1\rho = \varphi^{\ast}\varphi_2^{\ast},
\end{split}
\end{equation*}
\begin{equation}\label{D2}
\begin{split}
   \rho\varphi_1^{\ast}\varphi = \varphi_2,
\end{split}
\end{equation}
\begin{equation}\label{D3}
\begin{split}
   \varphi_2^{\ast}\sigma = \varphi^{\ast}\varphi_2^{\ast},
\end{split}
\end{equation}
and
\begin{equation}\label{D4}
\begin{split}
   \sigma\varphi_1^{\ast}\varphi = \varphi_1^{\ast}\varphi^{\ast}\varphi.
\end{split}
\end{equation}
By calculation,
we have
\begin{equation*}
\begin{split}
   \varphi_2 &~\stackrel{(\ref{D2})}{=} \rho\varphi_1^{\ast}\varphi =\rho\varphi_1^{\ast}\varphi_1\varphi_2,
\end{split}
\end{equation*}
\begin{equation*}
\begin{split}
   \varphi_2\varphi_1\varphi_2 = \varphi_2\varphi = (\varphi^{\ast}\varphi_2^{\ast})^{\ast} &~\stackrel{(\ref{D3})}{=} (\varphi_2^{\ast}\sigma)^{\ast} = \sigma^{\ast}\varphi_2,
\end{split}
\end{equation*}
and
\begin{equation*}
\begin{split}
   \sigma\varphi_1^{\ast}\varphi_1\varphi_2 = \sigma\varphi_1^{\ast}\varphi &~\stackrel{(\ref{D4})}{=} \varphi_1^{\ast}\varphi^{\ast}\varphi = \varphi_1^{\ast}\varphi^{\ast}\varphi_1\varphi_2.
\end{split}
\end{equation*}
Since $\varphi_2$ is monic,
then $1_Z = \rho\varphi_1^{\ast}\varphi_1$,
$\varphi_2\varphi_1 = \sigma^{\ast}$ and $\sigma\varphi_1^{\ast}\varphi_1 = \varphi_1^{\ast}\varphi^{\ast}\varphi_1$.
Thus $\varphi_2\varphi_1$ is invertible follows from $\varphi_2\varphi_1 = \sigma^{\ast}$.
Moreover,
$\sigma\varphi_1^{\ast}\varphi_1 = \varphi_1^{\ast}\varphi^{\ast}\varphi_1$ implies $\varphi_1^{\ast}\varphi_1 = \sigma^{-1}\varphi_1^{\ast}\varphi^{\ast}\varphi_1$,
then
\begin{equation*}
\begin{split}
      1_Z
   &~= \rho\varphi_1^{\ast}\varphi_1 = \rho\sigma^{-1}\varphi_1^{\ast}\varphi^{\ast}\varphi_1 = \rho\sigma^{-1}(\sigma^{-1}\sigma)\varphi_1^{\ast}\varphi^{\ast}\varphi_1 = \rho\sigma^{-2}(\sigma^{\ast})^{\ast}\varphi_1^{\ast}\varphi^{\ast}\varphi_1\\
   &~= \rho\sigma^{-2}(\varphi_2\varphi_1)^{\ast}\varphi_1^{\ast}\varphi^{\ast}\varphi_1 = \rho\sigma^{-2}\varphi_1^{\ast}(\varphi^{\ast})^2\varphi_1 = \rho\sigma^{-2}(\sigma^{-1}\sigma)\varphi_1^{\ast}(\varphi^{\ast})^2\varphi_1\\
   &~= \rho\sigma^{-3}(\sigma^{\ast})^{\ast}\varphi_1^{\ast}(\varphi^{\ast})^2\varphi_1 = \rho\sigma^{-3}(\varphi_2\varphi_1)^{\ast}\varphi_1^{\ast}(\varphi^{\ast})^2\varphi_1\\
   &~= \rho\sigma^{-3}\varphi_1^{\ast}(\varphi^{\ast})^3\varphi_1 = \cdots = \rho\sigma^{-n}\varphi_1^{\ast}(\varphi^{\ast})^n\varphi_1,
\end{split}
\end{equation*}
hence $(\varphi^{\ast})^n\varphi_1$ is left invertible.
\end{proof}

The following theorem is a corresponding result for dual core inverse.

\begin{thm} \label{dual-core-inverse-1}
Let $\varphi : X \rightarrow X$ be a morphism of a category with involution.
If $(\varphi_1, Z, \varphi_2)$ is an (epic, monic) factorization of $\varphi$ through $Z$,
then the following statements are equivalent: \\
(i) $\varphi$ has a dual core inverse with respect to $\ast$;\\
(ii) $\varphi_2\varphi^{\ast} : Z \rightarrow X$ is right invertible and $\varphi_2\varphi_1 : Z \rightarrow Z$ is invertible;\\
(iii) $\varphi_2(\varphi^{\ast})^n : Z \rightarrow X$ and $\varphi_2\varphi_1 : Z \rightarrow Z$ are both right invertible for any positive integer $n\geqslant 2$;\\
(iv) $(\varphi_1, Z, \varphi_2(\varphi^{\ast})^2)$,
$(\varphi\varphi^{\ast}\varphi_2^{\ast}, Z, \varphi_1^{\ast})$ and $(\varphi\varphi_2^{\ast}, Z, \varphi_1^{\ast}\varphi^{\ast})$ are all essentially unique (epic, monic) factorizations of $\varphi(\varphi^{\ast})^2$ through $Z$.
\end{thm}

In \cite{Chen} and \cite{Li},
authors characterized  the coexistence of core inverse and dual core inverse of a regular element by units in a $\ast$-ring.
And we give characterizations of the coexistence of core inverse and dual core inverse of a morphism with an (epic, monic) factorization in a category $\mathscr{C}$.

\begin{thm} \label{dual-core-inverse-8}
Let $\varphi : X \rightarrow X$ be a morphism of a category with involution and $n\geqslant 2$ a positive integer.
If $(\varphi_1, Z, \varphi_2)$ is an (epic, monic) factorization of $\varphi$ through $Z$,
then the following statements are equivalent: \\
(i) $\varphi$ is both Moore-Penrose invertible and group invertible;\\
(ii) $\varphi$ is both core invertible and dual core invertible;\\
(iii) $(\varphi^{\ast})^n\varphi_1 : X \rightarrow Z$ is left invertible and $\varphi_2(\varphi^{\ast})^n : Z \rightarrow X$ is right invertible;\\
(iv) $\varphi^n\varphi_2^{\ast} : X \rightarrow Z$ is left invertible and $\varphi_1^{\ast}\varphi^n : Z \rightarrow X$ is right invertible.\\
In this case,
\begin{equation*}
\begin{split}
    \varphi^{\co}&~=\varphi^{n-1}\mu^{\ast}\varphi_1^{\ast},\\
    \varphi_{\co}&~=\varphi_2^{\ast}\nu^{\ast}\varphi^{n-1},\\
    \varphi^{\dagger}&~=\nu\varphi_2\varphi^{2n-1}\mu^{\ast}\varphi_1^{\ast},\\
    \varphi^{\#}&~=(\varphi^{n-1}\mu^{\ast}\varphi_1^{\ast})^2\varphi=\varphi(\varphi_2^{\ast}\nu^{\ast}\varphi^{n-1})^2,
\end{split}
\end{equation*}
where $\mu(\varphi^{\ast})^n\varphi_1=1_Z=\varphi_2(\varphi^{\ast})^n\nu$ for some $\mu : Z \rightarrow X$ and $\nu : X \rightarrow Z$.
\end{thm}

\begin{proof}
$(i)\Leftrightarrow (ii)$. Obviously.

$(ii)\Rightarrow (iii)$. It is clear by Theorem~\ref{co-inv.} and Theorem~\ref{dual-core-inverse-1}.

$(iii)\Rightarrow (ii)$. Suppose that $\mu(\varphi^{\ast})^n\varphi_1=1_Z=\varphi_2(\varphi^{\ast})^n\nu$ for some $\mu : Z \rightarrow X$ and $\nu : X \rightarrow Z$,
where $n\geqslant 2$ is a positive integer.
Then we have
\begin{equation*}
\begin{split}
      \varphi=\varphi_1\cdot1_Z\cdot\varphi_2 = \varphi_1(\mu(\varphi^{\ast})^n\varphi_1)\varphi_2 = \varphi_1\mu(\varphi^{\ast})^n\varphi
\end{split}
\end{equation*}
and
\begin{equation*}
\begin{split}
      \varphi=\varphi_1\cdot1_Z\cdot\varphi_2 = \varphi_1(\varphi_2(\varphi^{\ast})^n\nu)\varphi_2 = \varphi(\varphi^{\ast})^n\nu\varphi_2.
\end{split}
\end{equation*}
Hence, the conclusion is now a consequence of Lemma~\ref{Li2}.

$(ii)\Leftrightarrow (iv)$.
Since $\varphi^{\ast}$ exists and has an (epic, monic) factorization $(\varphi_2^{\ast}, Z, \varphi_1^{\ast})$,
and $\varphi$ is both core invertible and dual core invertible if and only if $\varphi^{\ast}$ is both core invertible and dual core invertible.
Therefore,
the conclusion is a consequence of the preceding argument.

The expressions can be deduced by Lemma~\ref{Li2}.
\end{proof}

Let $\mathbb{C}_{m, n}$ be the set of all $m\times n$ complex matrices.
In \cite{WL},
H.X. Wang and X.L. Liu showed us that if $A\in \mathbb{C}_n^{\mathrm{CM}}$ has a full-rank decomposition $A=BC$,
then $A^{\co}=B(CB)^{-1}(B^{\ast}B)^{-1}B^{\ast}$,
where $\mathbb{C}_n^{\mathrm{CM}}=\{A\in \mathbb{C}_{n, n}: \mathrm{rank}(A^2)=\mathrm{rank}(A)\}$.
We will show another derivation for this result as follow.

\begin{cor}\cite[Theorem $2.4$]{WL}
Let $A\in \mathbb{C}_n^{\mathrm{CM}}$ with $\mathrm{rank}(A)=r$.
If $A$ has a full-rank decomposition $A=BC$,
then
\begin{equation*}
\begin{split}
   A^{\co}=B(CB)^{-1}(B^{\ast}B)^{-1}B^{\ast}.
\end{split}
\end{equation*}
\end{cor}

\begin{proof}
Let $U=(B^{\ast}B)^{-1}((CB)^{\ast})^{-1}(CC^{\ast})^{-1}C$,
then
\begin{equation*}
\begin{split}
   U(A^{\ast})^2B
   &~=[(B^{\ast}B)^{-1}((CB)^{\ast})^{-1}(CC^{\ast})^{-1}C](A^{\ast})^2B\\
   &~=[(B^{\ast}B)^{-1}((CB)^{\ast})^{-1}(CC^{\ast})^{-1}C](BC)^{\ast}(BC)^{\ast}B\\
   &~=[(B^{\ast}B)^{-1}((CB)^{\ast})^{-1}(CC^{\ast})^{-1}C]C^{\ast}(CB)^{\ast}B^{\ast}B\\
   &~=I_r,
\end{split}
\end{equation*}
thus $U$ is a left inverse of $(A^{\ast})^2B$.
According to the proof $(iii) \Rightarrow (i)$ of  Theorem~\ref{co-inv.},
we deduce that $A^{\co}=A(BU)^{\ast}$.
Therefore,
\begin{equation*}
\begin{split}
   A^{\co}
   &~=A(BU)^{\ast}\\
   &~=A[B(B^{\ast}B)^{-1}((CB)^{\ast})^{-1}(CC^{\ast})^{-1}C]^{\ast}\\
   &~=(BC)C^{\ast}(CC^{\ast})^{-1}(CB)^{-1}(B^{\ast}B)^{-1}B^{\ast}\\
   &~=B(CB)^{-1}(B^{\ast}B)^{-1}B^{\ast}.
\end{split}
\end{equation*}
\end{proof}

Likewise,
we have the following result.

\begin{cor}
Let $A\in \mathbb{C}_n^{\mathrm{CM}}$ with $\mathrm{rank}(A)=r$.
If $A$ has a full-rank decomposition $A=BC$,
then
\begin{equation*}
\begin{split}
   A_{\co}=C^{\ast}(CC^{\ast})^{-1}(CB)^{-1}C.
\end{split}
\end{equation*}
\end{cor}

\section{Applications}\label{dc}
Let $R$ be a ring,
and let $_R\mathrm{Mod}$ be the category of $R$-modules and $R$-morphisms.
In \cite{PR2},
R. Puystjens and D.W. Robinson mentioned that associated with every morphism $\tau : M \rightarrow N$ of $_R\mathrm{Mod}$ are the $R$-modules $\mathrm{Im}\tau=M\tau=\{x\tau | x\in M\}$,
$\mathrm{Ker}\tau=\{x | x\tau=0\}$ and the $R$-morphisms $\tau_1 : M \rightarrow \mathrm{Im}\tau$,
$x \mapsto x\tau$,
and $\tau_2 : \mathrm{Im}\tau \rightarrow N$,
$x \mapsto x$.
In particular,
$(\tau_1, \mathrm{Im}\tau, \tau_2)$ is an (epic, monic) factorization of $\tau$ through the object $\mathrm{Im}\tau$,
which is herein called the standard factorization of $\tau$ in $_R\mathrm{Mod}$.

Now we consider the coexistence of the core inverse and dual core inverse of an $R$-morphism in the category of $R$-modules of a given ring $R$.

\begin{lem}\label{tau}
Let $\tau : M \rightarrow M$ be a morphism of $_R\mathrm{Mod}$ with standard factorization $(\tau_1, \mathrm{Im}\tau, \tau_2)$,
and let $n$ be a positive integer.
If the full subcategory determined by $M$ has an involution $\ast$,
then\\
(i) $(\tau^{\ast})^n\tau_1$ is epic if and only if $\mathrm{Im}(\tau^{\ast})^n\tau=\mathrm{Im}\tau$;\\
(ii) $\tau_2(\tau^{\ast})^n$ is monic if and only if $\mathrm{Ker}\tau(\tau^{\ast})^n=\mathrm{Ker}\tau$;\\
(iii) $\mathrm{Ker}(\tau^{\ast})^n\tau_1=\mathrm{Ker}(\tau^{\ast})^n\tau$;\\
(iv) $\mathrm{Im}\tau_2(\tau^{\ast})^n=\mathrm{Im}\tau(\tau^{\ast})^n$.
\end{lem}

\begin{proof}
$(i)$.
Assume that $(\tau^{\ast})^n\tau_1 : M \rightarrow \mathrm{Im}\tau$ is epic.
It is obvious that $\mathrm{Im}(\tau^{\ast})^n\tau\subseteq \mathrm{Im}\tau$,
so we only need to prove that $\mathrm{Im}\tau\subseteq \mathrm{Im}(\tau^{\ast})^n\tau$.
Since $(\tau^{\ast})^n\tau_1$ is epic,
if $z\in \mathrm{Im}\tau$,
then there is a $y\in M$ such that $z=y(\tau^{\ast})^2\tau_1$,
and
$$z=z\tau_2=(y(\tau^{\ast})^n\tau_1)\tau_2=y(\tau^{\ast})^n\tau\in \mathrm{Im}(\tau^{\ast})^n\tau.$$
Thus,
$\mathrm{Im}(\tau^{\ast})^n\tau=\mathrm{Im}\tau$.
Conversely,
Assume that $\mathrm{Im}(\tau^{\ast})^n\tau=\mathrm{Im}\tau$ and let $z\in \mathrm{Im}\tau$,
then there is a $y\in M$ such that
$$z=y(\tau^{\ast})^n\tau=y(\tau^{\ast})^n\tau_1\tau_2=y(\tau^{\ast})^n\tau_1.$$
That is to say,
$(\tau^{\ast})^n\tau_1$ is surjective as a function and hence is epic as an $R$-morphism.

$(ii)$.
Suppose that $\tau_2(\tau^{\ast})^n : \mathrm{Im}\tau \rightarrow M$ is monic.
It is easy to see that $\mathrm{Ker}\tau\subseteq \mathrm{Ker}\tau(\tau^{\ast})^n$,
hence,
we only need to show that $\mathrm{Ker}\tau(\tau^{\ast})^n\subseteq \mathrm{Ker}\tau$.
Since $\tau_2(\tau^{\ast})^n$ is monic,
for any $z\in \mathrm{Ker}\tau(\tau^{\ast})^n$,
we have $0=z\tau(\tau^{\ast})^n=z\tau_1\tau_2(\tau^{\ast})^n$,
thus $0=z\tau_1=z\tau_1\tau_2=z\tau$,
that is $z\in \mathrm{Ker}\tau$.
Conversely,
suppose $\mathrm{Ker}\tau(\tau^{\ast})^n=\mathrm{Ker}\tau$ and $z\tau_2(\tau^{\ast})^n=0$,
where $z\in \mathrm{Im}\tau$.
Since $\tau_1$ is epic,
there exists a $y\in M$ such that $z=y\tau_1$.
Therefore,
$$0=z\tau_2(\tau^{\ast})^n=(y\tau_1)\tau_2(\tau^{\ast})^n=y\tau(\tau^{\ast})^n,$$
which follows that $0=y\tau=y\tau_1\tau_2=y\tau_1=z$.
Hence,
$\tau_2(\tau^{\ast})^n$ is monic.

Part $(iii)$ follows from the fact that $\tau_2$ is an insertion and part $(iv)$ is a consequence of the fact that $\tau_1$ is epic.
\end{proof}

\begin{thm}\label{new}
Let $\tau : M \rightarrow M$ be a morphism of the category $_R\mathrm{Mod}$,
and let $n\geqslant 2$ be a positive integer.
If the full subcategory determined by $M$ has an involution $\ast$,
then the following statements are equivalent: \\
(i) $\tau$ is both core invertible and dual core invertible;\\
(ii) $\tau$ is both Moore-Penrose invertible and group invertible;\\
(iii) both $\mathrm{Ker}(\tau^{\ast})^n\tau$ and $\mathrm{Im}\tau(\tau^{\ast})^n$ are direct summands of $M$,
and $\mathrm{Im}(\tau^{\ast})^n\tau=\mathrm{Im}\tau$,
$\mathrm{Ker}\tau(\tau^{\ast})^n=\mathrm{Ker}\tau$;\\
(iv) both $\mathrm{Ker}\tau^n\tau^{\ast}$ and $\mathrm{Im}\tau^{\ast}\tau^n$ are direct summands of $M$,
and $\mathrm{Im}\tau^n\tau^{\ast}=\mathrm{Im}\tau^{\ast}$,
$\mathrm{Ker}\tau^{\ast}\tau^n=\mathrm{Ker}\tau^{\ast}$;\\
(v) $M=\mathrm{Ker}\tau\oplus \mathrm{Im}(\tau^{\ast})^n$,
$M=\mathrm{Ker}(\tau^{\ast})^n\oplus \mathrm{Im}\tau$;\\
(vi) $M=\mathrm{Ker}\tau^{\ast}\oplus \mathrm{Im}\tau^n$,
$M=\mathrm{Ker}\tau^n\oplus \mathrm{Im}\tau^{\ast}$.
\end{thm}

\begin{proof}
$(i)\Leftrightarrow (ii)$.
Clearly.

$(i)\Leftrightarrow (iii)$.
As is known that an epic morphism in $_R\mathrm{Mod}$ is left invertible if and only if its kernel is a direct summand of its domain.
(See for example \cite[p. 12]{DG}.)
In particular,
$(\tau^{\ast})^n\tau_1$ is left invertible if and only if $(\tau^{\ast})^n\tau_1$ is epic and $\mathrm{Ker}(\tau^{\ast})^n\tau_1$ is a direct summand of $M$.
Thus,
by (i) and (iii) in Lemma~\ref{tau},
$(\tau^{\ast})^n\tau_1$ is left invertible if and only if $\mathrm{Im}(\tau^{\ast})^n\tau=\mathrm{Im}\tau$ and $\mathrm{Ker}(\tau^{\ast})^n\tau$ is a direct summand of $M$.
In a similar way,
since a monic morphism in $\mathrm{Mod}_R$ is right invertible if and only if its image is a direct summand of its codomain.
then from (ii) and (iv) in Lemma~\ref{tau},
$\tau_2(\tau^{\ast})^n$ is right invertible if and only if $\mathrm{Ker}\tau(\tau^{\ast})^n=\mathrm{Ker}\tau$ and $\mathrm{Im}\tau(\tau^{\ast})^n$ is a direct summand of $M$.
Consequently,
we get the conclusion by Theorem~\ref{dual-core-inverse-8}.

$(i)\Leftrightarrow (iv)$.
Since $\tau$ is both core invertible and dual core invertible if and only if $\tau^{\ast}$ is both core invertible and dual core invertible.
Then,
we can get this conclusion by replacing $\tau$ with $\tau^{\ast}$ in the preceding argument.

$(i)\Rightarrow (v)$.
Given $\tau^{\co}$ and $\tau_{\co}$,
then $M=M(1_M-\tau\tau^{\co})\oplus M\tau\tau^{\co}$.
Clearly $M(1_M-\tau\tau^{\co})=\mathrm{Ker}\tau$.
Since
\begin{equation*}
\begin{split}
   \tau\tau^{\co}
    &~=(\tau\tau^{\co})^{\ast}=(\tau^{\co})^{\ast}\tau^{\ast}=(\tau\tau^{\co}\tau^{\co})^{\ast}\tau^{\ast}=((\tau^{\co})^2)^{\ast}(\tau^{\ast})^2\\
    &~=(\tau^{\co})^{\ast}(\tau\tau^{\co}\tau^{\co})^{\ast}(\tau^{\ast})^2=((\tau^{\co})^3)^{\ast}(\tau^{\ast})^3\\
    &~=\cdots=((\tau^{\co})^n)^{\ast}(\tau^{\ast})^n
\end{split}
\end{equation*}
and
$$(\tau^{\ast})^n=(\tau^{\ast})^{n-1}\tau^{\ast}=(\tau^{\ast})^{n-1}(\tau\tau^{\co}\tau)^{\ast}=(\tau^{\ast})^{n-1}\tau^{\ast}\tau\tau^{\co}, $$
then $M\tau\tau^{\co}=\mathrm{Im}(\tau^{\ast})^n$.
Thus $M=\mathrm{Ker}\tau\oplus \mathrm{Im}(\tau^{\ast})^n$.

In addition,
for any $z\in M$,
$z=(z-z\tau_{\co}\tau)+z\tau_{\co}\tau$,
where $z\tau_{\co}\tau\in \mathrm{Im}\tau$.
Now we show that $z-z\tau_{\co}\tau\in \mathrm{Ker}(\tau^{\ast})^n$.
Since
$$(\tau^{\ast})^n=(\tau\tau_{\co}\tau)^{\ast}(\tau^{\ast})^{n-1}=\tau_{\co}\tau\tau^{\ast}(\tau^{\ast})^{n-1}=\tau_{\co}\tau(\tau^{\ast})^n,$$
then $(z-z\tau_{\co}\tau)(\tau^{\ast})^n=z(\tau^{\ast})^n-z\tau_{\co}\tau(\tau^{\ast})^n=0$.
Let $y\in \mathrm{Ker}(\tau^{\ast})^n\cap \mathrm{Im}\tau$,
then $y(\tau^{\ast})^n=0$ and there exists an $x\in M$ such that $y=x\tau$.
Hence,
\begin{equation*}
\begin{split}
   y
   &~=x\tau=x(\tau\tau_{\co}\tau)=x\tau\tau^{\ast}(\tau_{\co})^{\ast}=x\tau\tau^{\ast}(\tau_{\co}\tau_{\co}\tau)^{\ast}\\
   &~=y(\tau^{\ast})^2(\tau_{\co}^2)^{\ast}=y(\tau^{\ast})^2(\tau_{\co}\tau_{\co}\tau)^{\ast}\tau_{\co}^{\ast}\\
   &~=y(\tau^{\ast})^3(\tau_{\co}^3)^{\ast}=\cdots=y(\tau^{\ast})^n(\tau_{\co}^n)^{\ast}=0.
\end{split}
\end{equation*}
Therefore,
we have $M=\mathrm{Ker}(\tau^{\ast})^n\oplus \mathrm{Im}\tau$.

$(v)\Rightarrow (iii)$.
Let $M=\mathrm{Ker}\tau\oplus \mathrm{Im}(\tau^{\ast})^n$.
Then,
for any $z\in M$,
$z=k+y(\tau^{\ast})^n$,
where $k\in \mathrm{Ker}\tau$.
Therefore,
$$z\tau=y(\tau^{\ast})^n\tau\in \mathrm{Im}(\tau^{\ast})^n\tau,$$
which implies $\mathrm{Im}\tau\subseteq \mathrm{Im}(\tau^{\ast})^n\tau$.
Hence,
$\mathrm{Im}\tau=\mathrm{Im}(\tau^{\ast})^n\tau$.
In addition,
if $y(\tau^{\ast})^n\tau=0$,
then
$$y(\tau^{\ast})^n\in \mathrm{Ker}\tau\cap \mathrm{Im}(\tau^{\ast})^n=\{0\},$$
thus $\mathrm{Ker}(\tau^{\ast})^n\tau\subseteq \mathrm{Ker}(\tau^{\ast})^n$.
Moreover,
$\mathrm{Ker}(\tau^{\ast})^n\tau=\mathrm{Ker}(\tau^{\ast})^n$

Likewise,
let $M=\mathrm{Ker}(\tau^{\ast})^n\oplus \mathrm{Im}\tau$,
then $\mathrm{Ker}\tau(\tau^{\ast})^n=\mathrm{Ker}\tau$ and $\mathrm{Im}(\tau^{\ast})^n=\mathrm{Im}\tau(\tau^{\ast})^n$.

$(vi)\Leftrightarrow (v)$.
We can get this conclusion immediately by replacing $\tau$ with $\tau^{\ast}$ in the statement $(v)$.
\end{proof}

\begin{rem}
It should be noted that when taking $n=1$,
Lemma~\ref{tau} is consistent with \cite[Lemma $4$]{PR2},
and the statements (iii),
(iv),
(v)and (vi) in Theorem~\ref{new} are all equivalent to that $\tau$ is Moore-Penrose invertible.
(See \cite[Theorem $4$]{PR2}.)
\end{rem}

\begin{cor}\label{high}
Let $R$ be a $\ast$-ring and $a\in R$ and $n\geqslant 2$ a positive integer,
then the following statements are equivalent: \\
(i) $a$ is both Moore-Penrose invertible and group invertible;\\
(ii) $a$ is both core invertible and dual core invertible;\\
(iii) $R=$ $^{\circ}\!a\oplus R(a^{\ast})^n$,
$R=$ $^{\circ}\!((a^{\ast})^n)\oplus Ra$;\\
(iv) $R=(a^{\ast})^{\circ}\oplus a^nR$,
$R=(a^n)^{\circ}\oplus a^{\ast}R$;\\
(v) $R=$ $^{\circ}\!(a^{\ast})\oplus Ra^n$,
$R=$ $^{\circ}\!(a^n)\oplus Ra^{\ast}$;\\
(vi) $R=a^{\circ}\oplus (a^{\ast})^nR$,
$R=((a^{\ast})^n)^{\circ}\oplus aR$.
\end{cor}

\begin{proof}
As is known that $(i)\Leftrightarrow (ii)$.
And $(ii)\Leftrightarrow (iii)\Leftrightarrow (v)$ follows from Theorem~\ref{new}.
When taking involution on statements $(iii)$ and $(v)$,
we obtain statements $(iv)$ and $(vi)$, respectively.
\end{proof}

\vspace{0.2cm} \noindent {\large\bf Acknowledgements}

This research is supported by the National Natural Science Foundation of China (No.11771076); the Scientific Innovation Research of College Graduates in Jiangsu Province (No.KYCX17\_0037).

\end{document}